\documentclass[a4paper,twoside,12pt]{article}
\usepackage{amssymb,amsmath,amsthm,latexsym}
\usepackage{amsfonts}
\usepackage{graphicx,caption,subcaption}
\usepackage[pdftex,bookmarks,colorlinks=false]{hyperref}
\usepackage[hmargin=1.2in,vmargin=1.2in]{geometry}
\usepackage[mathscr]{euscript}

\usepackage[auth-sc-lg,affil-sl]{authblk}
\def\ni{\noindent}
\def\N{\mathbb{N}}
\def\cP{\mathscr{P}}
\def\f{f^{\ast}}

\newtheorem{thm}{Theorem}[section]

\newtheorem{cor}[thm]{Corollary}
\newtheorem{defn}{Definition}[section]

\newtheorem{prob}{Problem}
\newtheorem{prop}[thm]{Proposition}

\title{\textbf{\sc A Study on the Product Set-Labeling of Graphs}}

\author{Sudev Naduvath}
\affil{\small Centre for Studies in Discrete Mathematics\\ Vidya Academy of Science \& Technology \\ Thalakkottukara P. O., \\ Thrissur - 680501, Kerala, India.\\ {\tt sudevnk@gmail.com}}

\date{}

\begin{document}
\maketitle

\begin{abstract}
Let $X$ be a non-empty ground set and $\cP(X)$ be its power set. A set-labeling (or a set-valuation) of a graph $G$ is an injective set-valued function $f:V(G)\to \cP(X)$ such that the induced function $\f:E(G) \to \cP(X)$ is defined by $\f(uv)=f(u)\ast f(v)$, where $f(u)\ast f(v)$ is a binary operation of the sets $f(u)$ and $f(v)$. A graph which admits a set-labeling is known to be a  set-labeled graph. A set-labeling $f$ of a graph $G$ is said to be a set-indexer of $G$ if the associated function $\f$ is also injective.  In this paper, we introduce a new notion namely product set-labeling of graphs as an injective set-valued function $f:V(G)\to \cP(\N)$ such that the induced edge-function $\f:V(G)\to \cP(\N)$ is defined as $\f(uv)=f(u)\ast f(v) \forall\ uv\in E(G)$, where $f(u)\ast f(v)$ is the product set of the set-labels $f(u)$ and $f(v)$, where $\N$ is the set of all positive integers and discuss certain properties of the graphs which admit this type of set-labeling.
\end{abstract}

\vspace{0.2cm}

\ni \textbf{Key words}: Set-labeling of graphs; product set-labeling of graphs; uniform product set-labeling of graphs; geometric product set-labeling of graphs.

\vspace{0.04in}

\ni \textbf{Mathematics Subject Classification}: 05C78.

\section{Introduction}

For all  terms and definitions, not defined specifically in this paper, we refer to \cite{BM1,JAG, FH,DBW}. Unless mentioned otherwise, all graphs considered here are simple, finite, undirected and have no isolated vertices.

Let $X$ be a non-empty set and $\cP(X)$ be its power set. A {\em set-labeling} (or a \textit{set-valuation}) of a graph $G$ is an injective function $f:V(G)\to \cP(X)$ such that the induced function $f^{\oplus}:E(G)\to \cP(X)$ is defined by $f^{\oplus}(uv)=f(u)\oplus f(v)~ \forall ~ uv\in E(G)$, where $\oplus$ is the symmetric difference of two sets.  A graph $G$ which admits a set-labeling is called a {\em set-labeled graph} (or a set-valued graph)(see \cite{BDA1}).  

A {\em set-indexer} of a graph $G$ is an injective function $f:V(G)\to \cP(X)$ such that the induced function $f^{\oplus}:E(G) \to \cP(X)$ is also injective. A graph $G$ which admits a set-indexer is called a {\em set-indexed graph} (see \cite{BDA1}).

Several types of set-valuations of graphs have been introduced in later literature and the properties and structural characteristics of such set-valued graphs have been studied in a rigorous manner. A relevant and important set-labeling in this context is the integer additive set-labeling of graphs which is defined as an injective set-valued function $f:V(G)\to \cP(X)$ such that the induced edge function $f^+:E(G)\to \cP(X)$ is defined by $f^+(uv)=f(u)+f(v) \forall\ uv\in E(G)$, where $X$ is a non-empty set of non-negative integers and $f(u)+f(v)$ is the sumset of the set-labels $f(u)$ and $f(v)$. Certain types of integer additive set-labeled graphs are studied in \cite{GS1,GS2,GS3,GS4}.

Motivated by these studies on different types of set-valuations of graphs, in this paper, we introduce a new type of set-labeling, namely product set-labeling of graphs and study the properties and characteristics of the graphs which admit this type of set-labeling.

\section{Product Set-Labeling of Graphs}

Let $A$ and $B$ be two sets of integers. Then, the \textit{product set} of $A$ and $B$, denoted by $A\ast B$, is the set defined by $A\ast B=\{ab: a\in A, b\in B\}$. Note that $A\ast \emptyset =\emptyset$ and $A\ast \{0\}=\{0\}$. Also, if either $A$ or $B$ is a countably infinite set, then their product set is also countably infinite. In view of these facts, we restrict our studies to the non-empty finite sets of positive integers.

Analogous to the corresponding results on sumsets of sets of integers (see \cite{MBN}), we have the following result on the cardinality of the product set of two sets of positive integers.

\begin{thm}\label{Thm-1}
If $A$ and $B$ are two non-empty finite sets of positive integers, then $|A|+|B|-1\le |A\ast B|\le |A|\,|B|$.
\end{thm}

\begin{thm}\label{Thm-2}
For any two sets $A$ and $B$ of positive integers, $|A\ast B|=|A|+|B|-1$ if and only if $A$ and $B$ are geometric progressions with the same common ratio.
\end{thm}

\ni Using the above mentioned concepts of product sets of sets of positive integers, we introduce the notion of the product set-labeling of a graph as given below.

\begin{defn}{\rm 
Let $\N$ be the set of all positive integers and $\cP(\N)$ be its power set. The \textit{product set-labeling} of a graph $G$ is an injective set-valued function $f:V(G)\to \cP(\N)$ such that the induced edge-function $\f:V(G)\to \cP(\N)$ is defined as $\f(uv)=f(u)\ast f(v) \forall\ uv\in E(G)$, where $f(u)\ast f(v)$ is the product set of the set-labels $f(u)$ and $f(v)$.  A graph $G$ which admits a product set-labeling is called a \textit{product set-labeled graph}. }
\end{defn}

\begin{defn}{\rm 
A \textit{product set-labeling} $f:V(G)\to \cP(\N)$ of a graph $G$ is said to be a \textit{product set-indexer} if the induced edge-function $\f:V(G)\to \cP(\N)$ defined by $\f(uv)=f(u)\ast f(v) \forall\ uv\in E(G)$ is also an injective function.  }
\end{defn}

The cardinality of the set-label of an element (a vertex or an edge) of $G$ is called \textit{label size} of that element. A product set-labeling $f$ of a graph $G$ is said to be a \textit{uniform product set-labeling} if all edges of $G$ have the same label size under $f$. In particular, a product set-labeling $f$ of a graph $G$ is said to be \textit{$k$-uniform} if $|\f(uv)|=k\ \forall\, uv\in E(G)$.

\vspace{0.25cm}

In view of Theorem \ref{Thm-1}, the bounds for the label size of edges of a product set-labeled graph $G$ is given by
\begin{equation}
|f(u)|+|f(v)|-1\le |\f(uv)|=|f(u)\ast f(v)|\le |f(u)|\,|f(v)|\ \forall\ uv\in E(G). \label{eqn-1}
\end{equation}

The product set-labelings which satisfy the bounds of this inequality are of special interest. If the cardinality of the vertex set-labels of $G$ under a product set-labeling $f$ attains the upper bound of the inequality \eqref{eqn-1}, then $f$ is called a \textit{strong product set-labeling} of $G$. Before proceeding to investigate the conditions for the existence of a strong product set-labeling, we require the following notion.

\begin{defn}{\rm 
Let $A$ be a non-empty set of positive integers. Then the \textit{quotient set} of the set $A$, denoted by $Q_A$, is defined as $Q_A=\{\frac{a}{b}:a,b\in A, a\ge b\}$. That is, $Q_A$ is the set of all rational numbers, greater than or equal to $1$, which is formed by the elements of the set $A$.}
\end{defn}

In view of this notion, we establish a necessary and sufficient condition for a graph $G$ to admit a strong product set-labeling in the following theorem.

\begin{thm}\label{thm-3}
A product set-labeling $f$ of a graph $G$ is a strong product set-labeling if and only if the quotient sets of the set-labels of any pair of adjacent vertices of $G$ are disjoint.
\end{thm}
\begin{proof}
Let $f$ be a product set-labeling of a graph $G$ and let $u$ and $v$ be any two adjacent vertices in $G$.

First, assume that $f$ is a strong product set-labeling of $G$. Then, we have $|f(u)\ast f(v)|= |f(u)|\,|f(v)| \forall\ uv\in E(G)$. This is possible only when $ac\ne bd$ for any two distinct elements $a,b\in f(u)$ and any two distinct elements $c,d\in f(v)$. That is, $\frac{a}{b}\ne \frac{c}{d}$. Since $\frac{a}{b} \in Q_{f(u)}$ and $\frac{c}{d} \in Q_{f(v)}$, we have $Q_{f(u)}\cap Q_{f(v)}=\emptyset$. 

If possible let $f$ is not a strong product set-labeling of $G$. Then $|f(u)\ast f(v)|< |f(u)|\,|f(v)|$ for some $uv\in E(G)$. That is, there exist at least two elements $a,b\in f(u)$ and at least two elements $c,d\in f(v)$ such that $ac=bd$. That is, $\frac{a}{b}=\frac{c}{d}$. Hence, $Q_{f(u)}\cap Q_{f(v)}\ne \emptyset$. This completes the proof.
\end{proof}

By saying that a set is a geometric progression, we mean that the elements of that set is in geometric progression. If the context is clear, the common ratio of the set-label of an element (a vertex or an edge) in $G$ may be called the \textit{common ratio of that element}.  

In view of Theorem \ref{Thm-2}, we also note that the vertex set-labels in $G$, under the product set-labeling $f$ which attains the lower bound the inequality \eqref{eqn-1}, are geometric progressions having the same common ratio. This fact creates much interest in investigating the set-labels of the elements of $G$ which are geometric progressions. Hence we have the following notion.

\begin{defn}{\rm 
A product set-labeling $f:V(G)\to \cP(\N)$ of a graph $G$ is said to be a \textit{geometric product set-labeling} if the set-labels of all elements (vertices and edges) of $G$ with respect to $f$ are geometric progressions.}
\end{defn}

The following theorem discusses the conditions required for a product set-labeling $f$ of a graph $G$ to be a geometric product set-labeling of $G$.

\begin{thm}\label{Thm-3}
A product set-labeling $f:V(G)\to \cP(\N)$ of a graph $G$ is a geometric product set-labeling of $G$ if and only if for every edge of $G$, the common ratio of one end vertex is a positive integral power of the common ratio of the other end vertex, where this power is less than or equal to the label size of the end vertex having smaller common ratio. 
\end{thm}
\begin{proof}
Let $f$ be a product set-labeling of a graph $G$ under which every vertex set-label is a geometric progression and let $u$ and $v$ be any two adjacent vertices of $G$. Let $r_u$ and $r_v$ be the common ratios of $u$ and $v$ respectively such that $r_u\le r_v$. Let the set-labels of $u$ and $v$ be given by $f(u)=\{a_i=a(r_u)^{i-1}:a\in \N; 0\le i\le |f(u)|=m\}$ and $f(u)=\{b_j=b(r_v)^{j-1}:b\in \N; 0\le j\le |f(v)|=n\}$. Now, consider the following sets.  
\vspace{-0.75cm}

\begin{alignat*}{4}
A_0 & = & f(u)\ast \{b_0\} & = & \{ab, abr_u, abr_u^2, \ldots, ab(r_u)^{m-1}\}\\
A_1 & = & f(u)\ast \{b_1\} & = & \{abr_v, abr_ur_v, abr_u^2r_v, \ldots, ab(r_u)^{m-1}r_v\}\\
\vdots & & \vdots & & \vdots\\ 
A_j & = & f(u)\ast \{b_j\} & = & \{abr_v^{j}, abr_ur_v^{j}, abr_u^2r_v^{j}, \ldots, ab(r_u)^{m-1}r_v^{j}\}\\
\vdots & & \vdots & & \vdots\\
A_{n-1} & = & f(u)\ast \{b_{n-1}\} & = & \{abr_v^{n-1}, abr_ur_v^{n-1}, abr_u^2r_v^{n-1}, \ldots, ab(r_u)^{m-1}r_v^{n-1}\}
\end{alignat*}

Here we can see that $\f(uv)=\bigcup\limits_{j=0}^{n-1} A_j$. 

Now assume that $r_v=(r_u)^k$, for some positive integer $k\le |f(u)|=m$. Then, either some of the initial elements of the set $A_{j+1}$ coincides with some final elements of $A_j$ or the ratio between the first element of $A_{j+1}$ and the final element of $A_j$ is $r_u$, for $0\le j\le n-1$. In both cases $A_j\cup A_{j+1}$ is a geometric progression for all $0\le j\le n-1$. Hence $\f(uv)$ is a geometric progression for all edge $uv\in E(G)$ and hence $f$ is a geometric product set-labeling of $G$.

If $r_v=(r_u)^k$ and $k\ge |f(u)|$, then for $0\le j\le n-1$, we note the following facts 

\begin{enumerate}\itemsep0mm
	\item[(i)] $A_j$ and $A_{j+1}$ are geometric progressions with the same common difference $r_u$,
	\item[(ii)] $A_j\cap A_{j+1}=\emptyset$,
	\item[(iii)] $A_j\cup A_{j+1}$ is not a geometric progression, as the ratio between the first element of $A_{j+1}$ and the final element of $A_j$ is not equal to $r_u$.
\end{enumerate}
Therefore, in this case, $\f(uv)$ is not a geometric progression and hence $f$ is not a geometric product set-labeling of $G$.

Now consider the case that $r_v\ne (r_u)^k$ for any positive integer $k$. Then, $A_j$ and $A_{j+1}$ are geometric progressions with different common ratios and hence it is clear that $A_j\cup A_{j+1}$ is not a geometric progression. Therefore, in this case also, $\f(uv)$ is not a geometric progression and hence $f$ is not a geometric product set-labeling of $G$. This compltes the proof.
\end{proof}

The following result describes a necessary and sufficient condition for a complete graph to admit a geometric product set-labeling.

\begin{cor}\label{Cor-3}
A complete graph $K_n$ admits a geometric product set-labeling if and only if the common ratio of every vertex is either an integral power or a root of the common ratios of all other vertices of $K_n$.
\end{cor}
\begin{proof}
Since any two vertices in $K_n$ are adjacent to each other, the proof follows as an immediate consequence of Theorem \ref{Thm-3}.
\end{proof}

The \textit{characteristic index} of an edge $e=uv$ of a product set-labeled graph $G$ is the number $k\ge 1$, such that $r_v=(r_u)^k$, where $r_u$ and $r_v$ are the common ratios of the set-labels of the vertices $u$ and $v$ (or equivalently, the common ratios of $u$ and $v$) respectively.  

The following result discusses the label size of the edges of a graph $G$ which admits a geometric product set-labeling. 

\begin{prop}\label{Prop-3}
Let $f$ be a geometric product set-labeling of a graph $G$ and let $u$ and $v$ be two adjacent vertices in $G$ with the common ratios $r_u$ and $r_v$ such that $r_u\le r_v$. Then, the label size of the edge $uv$ is given by $|\f(uv)|=|f(u)|+k\left(|f(v)|-1\right)$, where $k$ is the characteristic index of the edge $uv$.
\end{prop}
\begin{proof}
Since $f$ is a geometric product set-labeling of $G$, by Theorem \ref{Thm-3}, for any adjacent vertices $u$ and $v$ with the common ratios $r_u$ and $r_v$ such that $r_u\le r_v$, we have $r_v=(r_u)^k$, where $k$ is a positive integer less than or equal to $|f(u)|$. Let $f(u)=\{ar_u^{i-1}: a\in \N; 0\le i\le |f(u)|\}$ and $f(v)=\{br_v^{j-1}: b\in \N; 0\le j\le |f(v)|\}$. Then, we have

\vspace{-0.75cm}

\begin{eqnarray*}
\f(uv)& = & \{abr_u^{i-1}r_v^{j-1}: a,b\in \N; 0\le i\le m, 0\le j\le n\}\\
& = & \{abr_u^{i-1}{(r_u)^k}^{j-1}: a,b\in \N; 0\le i\le m, 0\le j\le n\}\\
& = & \{abr_u^{(i-1)+k(j-1)}: a,b\in \N; 0\le i\le m, 0\le j\le n\}.
\end{eqnarray*}
Therefore, $|\f(uv)|=m+k(n-1)=|f(u)|+k\left(|f(v)|-1\right)$.
\end{proof}

The following theorem describes a necessary and sufficient condition for a geometric product set-labeling of a graph $G$ to be a strong product set-labeling of $G$.

\begin{thm}\label{Thm-4}
A geometric set-labeling $f$ of a graph $G$ is a strong product set-labeling of $G$ if and only if the characteristic index of every edge of $G$ is equal to the label size of its end vertex having smaller common ratio.
\end{thm}
\begin{proof}

Let $f$ be a geometric product set-labeling of $G$ and let $u$ and $v$ be two adjacent vertices of $G$ with common ratios $r_u$ and $r_v$ respectively such that $r_u\le r_v$.  

First, assume that $f$ is also a strong product set-labeling of $G$. Then, we have $|\f(uv)|= |f(u)|\,|f(v)|\ \forall\, uv\in E(G)$. But, by Proposition \ref{Prop-3}, we have $|\f(uv)|=|f(u)|+k\left(|f(v)|-1\right)$. Therefore, from the above two equations, we get
\begin{eqnarray*}
|\f(uv)|& = & |f(u)|+k\left(|f(v)|-1\right)\\
i.e, ~~ |f(u)|\,|f(v)|& = & |f(u)|+k\left(|f(v)|-1\right)\\
\therefore  ~~  k & = & \frac{|f(u)|\,|f(v)|-|f(u)|}{\left(|f(v)|-1\right)}\\ 
& = & |f(u)|.
\end{eqnarray*} 

Conversely assume that the characteristic index of every edge of $G$ is equal to the label size of its end vertex having smaller common ratio. That is, let $k=|f(u)|$. Then, we have 

\vspace{-0.75cm}

\begin{eqnarray*}
|\f(uv)|& = & |f(u)|+k\left(|f(v)|-1\right) \text{for all}\ uv\in E(G)\\
& = & |f(u)|+|f(u)|\left(|f(v)|-1\right)\\
& = & |f(u)|\,|f(v)|
\end{eqnarray*}
Therefore, $f$ is a strong product set-labeling of $G$. This completes the proof.
\end{proof}

When the set-labels of two adjacent vertices are geometric progressions with the same common ratio, then the characteristic index of the edge between them is $1$. Invoking this fact, we define a particular type of geometric product set-labeling as follows.

\begin{defn}{\rm 
An \textit{isogeometric product set-labeling} of a graph $G$ is a product set-labeling of $G$ with respect to which the set-labels of all elements of $G$ are geometric progressions with the same common ratio. } 
\end{defn} 

In view of Theorem \ref{Thm-2}, we note that for any graph $G$ which admits an isogeometric product set-labeling, the label size of every edge is one less than the sum of the label sizes of its end vertices and also note that the characteristic index of every edge of $G$ is $1$.

The following is an obvious result on the admissibility of an isogeometric product set-labeling by any given graph.

\begin{thm}\label{Thm-6}
Every graph admits an isogeometric product set-labeling.
\end{thm}
\begin{proof}
Let $V=\{v_1,v_2,v_3,\ldots,v_n\}$ be the vertex set of the given graph $G$. Choose two sets  $A=\{a_i\in \N : 1\le i\le |V|=n\}$ and $B=\{m_i\in \N: 1\le i\le n\}$. Now label vertices of $G$ by the geometric progression $f(v_i)=\{a_i,a_ir, a_ir^2,\ldots,a_ir^{m_i-1}\}$; $1\le i\le n$, where $r$ is a positive integer greater than $1$. Then, the set-label of any edge $v_iv_j$ in $G$ is given by $\f(v_iv_j)=\{a_ia_j, a_ia_jr, a_ia_jr^2,\ldots,a_ia_jr^{m_i+m_j-2}\}$, which is also a geometric progression with the common ratio $r$. That is, the set-label of all elements of $G$ are geometric progressions with the same common ratio $r$. Hence, $f$ is an isogeometric product set-labeling of $G$.  
\end{proof}

In the following theorem, we discuss the condition required for an isogeometric product set-labeling of a graph $G$ to be a uniform product set-labeling of $G$.

\begin{thm}\label{Thm-7}
An isogeometric product set-labeling of a connected graph $G$ is a uniform product set-labeling if and only if any one of the following conditions holds.
\begin{enumerate}\itemsep0mm
	\item[(i)] the label size of all vertices of $G$ are equal.
	\item[(ii)] $G$ is bipartite with label size of vertices in the same partition are equal.
\end{enumerate}
\end{thm}
\begin{proof}
Let $f$ be an isogeometric product set-labeling of a given graph $G$. If $|f(v)|=m$, a positive integer $m$ for all $v\in V(G)$, then all edges of $G$ has the label size $2m-1$. If there exist some vertices in $V(G)$ such that $|f(v)|\ne m$, then assume that $G$ is a bipartite graph with bipartition $(X,Y)$ such that all vertices in $X$ have the label size $m$ and all vertices in $Y$ have the label size $n$. Here, by Proposition \ref{Prop-3}, all edges of $G$ have the label size $m+n-1$. In both cases, $f$ is a uniform product set-labeling of $G$. 

Conversely, assume that the isogeometric product set-labeling $f$ is also a uniform product set-labeling of $G$. If the label size of all vertices of $G$ are equal, then the proof is complete. Hence assume otherwise. Let $u$ be an arbitrary vertex of $G$ which has the label size $m$. Since $f$ is a uniform geometric product set-labeling, all vertices $v$ in the neighbouring set $N(u)$ of the vertex $u$ must have the same label size, say $n$. Using the same argument, all vertices in the neighbouring set of $N(u)$ must have the label size $m$. Since $G$ is a connected graph, the vertex set $V(G)$ can be partitioned in to two sets such that the vertices in the first partition have the label size $m$ the vertices in the other partition have the label size $n$. Since $m\ne n$, no two vertices in the same partition are adjacent also. Hence, $G$ is a bipartite graph with the vertices in the same partition have same label size.
\end{proof}

We have already noticed that the characteristic index of all edges of a graph which admits an isogeometric product set-labeling is $1$. But, In general, the characteristic indices of all edges of a geometric product set-labeled graph need not be the same. This fact creates a lot of interest in studying the structural properties of a geometric product set-labeled graph, all whose edges have the same characteristic index greater than $1$. Hence we have the following notion. 

\begin{defn}{\rm 
A geometric product set-labeling of a  graph $G$ is said to be a \textit{like-geometric product set-labeling} if all edges have the same characteristic index $k>1$. }
\end{defn}

The following theorem discusses a necessary and sufficient condition for a graph $G$ to admit a like-geometric product set-labeling.

\begin{thm}\label{Thm-5}
A graph $G$ admits a like-geometric product set-labeling if and only if it is bipartite.
\end{thm}
\begin{proof}
First, assume that $G$ is a bipartite graph with a bipartition $(X,Y)$. Let $X=\{v_i: 1\le i\le m\}$ and $Y=\{u_j: 1\le j\le n\}$, where $m,n \in \N$.  Choose the sets $M_1=\{m_i\in \N :m_i\ge 2, 1\le i\le |X|=m\}$, $M_2=\{a_i\in \N : 1\le i\le m\}$, $N_1=\{n_j\in \N : n_j\ge 2, 1\le j\le |Y|=n\}$ and $N_2=\{b_j\in \N : 1\le j\le n\}$. Let $k=\min\{m_i:1\le i\le m\}$ and choose two positive integers $r$ and $s$ such that $s=r^k$. Now, define a product set-labeling $f$ on $G$ which assigns each vertex $v_i$ of $X$ to a geometric progression $f(v_i)=\{a_i,a_ir,a_ir^2,\ldots, a_ir^{m_i-1}\}; 1\le i\le |X|$ and each vertex $u_j$ of $Y$ to a geometric progression $f(u_j)=\{b_j,b_js,b_js^2,\ldots, b_js^{n_j-1}\}; 1\le j\le |Y|$. Then, by Theorem \ref{Thm-3}, for every edge $v_iu_j$ in $G$, if exists, the set-label $\f(v_iu_j)$ is a geometric progression with common ratio $r$ and the characteristic index of every edge in $G$ will be $k$. Hence, the function $f$ is a like-geometric product set-labeling of $G$.

Next, assume that $G$ is not a bipartite graph. Then, $G$ contains at least one odd cycle. Let $C_n$ be such an odd cycle in $G$. Now, choose two positive integers $r$ and $s$ such that $s=r^k$. Label the vertices of $C_n$ with odd subscripts by distinct geometric progressions with common ratio $r$ and label the vertices with even subscripts by distinct geometric progressions with common ratio $s$. Then, all edges except $v_nv_1$ attain the characteristic index $k$ and the edge $v_nv_1$ has the characteristic index $1$. In all other labeling of the vertices of $G$ with geometric progressions such that the maximum number of edges attains the characteristic index $k$, we can see that at least one edge of $C_n$ has the characteristic index $k^q$, for some positive integer $q\ne 1$. In all these cases, it is to be noted that $f$ is not a like-geometric product set-labeling of $G$. This completes the proof. 
\end{proof}

The following proposition provides the condition required for a like-geometric product set-labeling of a graph $G$ to be a uniform product set-labeling of $G$.

\begin{thm}\label{Thm-8}
A like-geometric product set-labeling of a (bipartite) graph $G$ is a uniform product set-labeling of $G$ if and only if the vertices in the same partition of $G$ have the same label size.
\end{thm}
\begin{proof}
Let $f$ be a like-geometric product set-labeling of a connected graph $G$. Then, by Theorem \ref{Thm-5}, $G$ is bipartite. Let $(X,Y)$ be a bipartition of $G$. 

First, let all vertices in $X$ have the same label size, say $m$ and all vertices in $Y$ have the same label size, say $n$. Then, by Proposition \ref{Prop-3}, the label size of all edges of $G$ is $m+k(n-1)$, where $k\le n$. Hence, $f$ is a uniform product set-labeling of $G$. 

Now, assume that $f$ is also a uniform product set-labeling of $G$. Then, exactly as explained in the converse part of the proof of Theorem \ref{Thm-7}, we can partition the vertex set of $G$ in two subsets $X$ and $Y$ such that all vertices in $X$ have the same label size, say $m$ and all vertices in $Y$ have the same set-label, say $n$ and such that no two vertices in the same partition are adjacent to each other. This completes the proof.   
\end{proof}

A necessary and sufficient condition for a like-geometric product set-labeling of a graph $G$ to be a strong product set-labeling of $G$ is provided in the following result.

\begin{thm}\label{Thm-9}
A like-geometric product set-labeling $f$ of a graph $G$ is a strong product set-labeling of $G$ if and only if all vertices in one partition have the same label size. 
\end{thm}
\begin{proof}
Let $f$ be a like-geometric product set-labeling of $G$. Clearly, by Theorem \ref{Thm-5}, $G$ is bipartite. Let $(X,Y)$ be a bipartition of $G$. Without loss of generality, label all vertices in $X$ by distinct geometric progressions of same cardinality, say $m$ and the same common ratio $r$, where $r$ is any positive integer greater than $1$. Now label the vertices in $Y$ by distinct geometric progressions with common ratio $r^m$. Then, by Theorem \ref{Thm-4}, $|\f(uv)|=|f(u)|\,|f(v)|\ \forall\, uv\in E(G)$. Therefore, $f$ is a strong product set-labeling of $G$.

Conversely, assume that $f$ is a strong product set-labeling of $G$. Then, the characteristic index $k$ of every edge of $G$ is equal to the cardinality of the set-label of its end vertex having smaller common ratio. Since $f$ is a like-geometric product set-labeling, the characteristic index of every edge of $G$ is the same and is equal to the minimum label size of the vertices having smaller common ratio. Hence, the label size of all vertices in the corresponding partition are the same. This completes graph.
\end{proof}

In view of the above two theorems, we have the following result.

\begin{cor}\label{Cor-6}
A like-geometric product set-labeling $f$ of a graph $G$ is a strongly uniform product set-labeling of $G$ if and only if all vertices in the same partition have the same label size. 
\end{cor}
\begin{proof}
The proof immediately follows from Theorem \ref{Thm-8} and Theorem \ref{Thm-9}, by taking the value $k=m$ in the respective proofs.  
\end{proof}

\section{Conclusion}

In this paper, we have discussed the characteristics and properties of the graphs which admit different types of product set-labeling. There are several open problems in this area. Some of the open problems that seem to be promising for further investigations are following.

\begin{prob}{\rm 
Characterise the product set-labeled graphs whose vertex set-labels are geometric progressions but the edge set-labels are not.}
\end{prob}

\begin{prob}{\rm 
Characterise the product set-labeled graphs whose edge set-labels are geometric progressions but the vertex set-labels are not.}
\end{prob}

\begin{prob}{\rm 
Discuss the conditions required for an arbitrary geometric product set-labeling of a graph to be a uniform product set-labeling of $G$.}
\end{prob}

\begin{prob}{\rm 
Discuss the admissibility of different types of product set-labelings by different graph operations, graph products and graph powers. }
\end{prob}

\begin{prob}{\rm 
Characterise the product set-labeled graphs in which the label size of its edges are equal to the label sizes of one or both of their end vertices. }
\end{prob}

Further studies on other characteristics of product set-labeled graphs corresponding to different types of product set-labelings are also interesting and challenging. All these facts highlight the scope for further studies in this area. 

\section*{Acknowledgement}

The author would like to dedicate this work to Prof. (Dr.) T. Thrivikraman, who has been his mentor, motivator and the role model in teaching as well as in research.

\end{document}